\documentclass[11pt]{amsart}

\usepackage{amsmath}
\usepackage{amsthm}
\usepackage{amsbsy}
\usepackage{amssymb}
\usepackage{mathtools}
\usepackage{calrsfs}
\usepackage{comment}
\usepackage[shortlabels]{enumitem}
\usepackage[normalem]{ulem}
\usepackage{amsrefs}
\usepackage{tikz}
\usetikzlibrary{positioning,patterns,calc}
\usepackage{mathrsfs}
\DeclareMathAlphabet{\mathpzc}{OT1}{pzc}{m}{it}
\usepackage[utf8]{inputenc}
\usepackage{hyperref}

\newtheorem{theorem}{Theorem}[section]
\newtheorem{lemma}[theorem]{Lemma}
\newtheorem{remark}[theorem]{Remark}
\newtheorem{proposition}[theorem]{Proposition}
\newtheorem{corollary}[theorem]{Corollary}

\newtheorem{definition}[theorem]{Definition}
\newtheorem*{theorem*}{Theorem}

\makeatletter
\@addtoreset{claim}{theorem}
\makeatother

\theoremstyle{plain}

\theoremstyle{plain}

\usepackage{enumitem}
\usepackage{lipsum}
\setlist[itemize]{leftmargin=*}

\def\bpf{\begin{proof}}
	\def\epf{\end{proof}}
\def\be{\begin{equation}}
	\def\ee{\end{equation}}
\def\bea{\begin{eqnarray}}
	\def\eea{\end{eqnarray}}
\def\bt{\begin{theorem}}
	\def\et{\end{theorem}}
\def\bl{\begin{lemma}}
	\def\el{\end{lemma}}
\def\br{\begin{remark}}
	\def\er{\end{remark}}
\def\bc{\begin{corollary}}
	\def\ec{\end{corollary}}
\def\bd{\begin{definition}}
	\def\ed{\end{definition}}
\def\bp{\begin{proposition}}
	\def\ep{\end{proposition}}

\def\be0{\partial_0 E}
\def\be1{\partial_1 E}

\title{CMC hypersurface with finite index in hyperbolic space $\mathbb{H}^4$ }
\author{Han Hong}
\address{Department of Mathematics and Statistics\\ Beijing Jiaotong University\\ Hai Dian District\\ Beijing\\ China 100044}
\email{hanhong@bjtu.edu.cn}

\begin{document}

	\begin{abstract}
		In this paper, we prove that there are no complete noncompact constant mean curvature hypersurfaces with the mean curvature $H>1$, finite index and finite topology in hyperbolic space $\mathbb{H}^4$. A more general nonexistence result can be proved in  a $4$-dimensional Riemannian manifold with certain curvature conditions. We also show that $4$-manifold with $\operatorname{Ric}>1$ does not contain any complete noncompact  minimal stable hypersurface with finite topology.
		
		The proof relies on the $\mu$-bubble initially introduced by Gromov and further developed by Chodosh-Li-Stryker in the context of stable minimal hypersurfaces.
	\end{abstract}
	
	\maketitle
	
	\section{Introduction}\label{introduction}
	The well-known Bernstein problem has been completely settled due to works of W. Fleming \cite{Flembing62}, De Giorgi \cite{De65}, F. Almgren \cite{Almgren66}, J. Simons \cite{Simons68}, and E. Bombieri, De Giorgi, and E. Giusti \cite{BDG69}. It states that a complete minimal graph over $\mathbb{R}^{n}$ for $2\leq n\leq 6$ must be hyperplane. As a natural generalization, it was proved by D. Fischer-Colbrie and R. Schoen \cite{Fischer-Colbrie-Schoen-The-structure-of-complete-stable}, M. do Carmo and C. Peng \cite{doCarmo-Peng} and A. Pogorelov \cite{Pogorelov81} independently that the only two-sided stable complete minimal surfaces in $\mathbb{R}^3$ are planes.  R. Schoen, L. Simon and S. Yau \cite{SSY75} proved the higher dimensional stable Bernstein theorem for $\mathbb{R}^{n+1}$ for $n\leq 5$ under the Euclidean volume growth condition. Note that recently C. Bellettini generalized it to $\mathbb{R}^7$ in \cite{bellettini}. 
	
	We recall that a two-sided minimal immersion $M^n$ into a Riemannian manifold $X^{n+1}$ is stable if for any compactly supported function $\varphi$ on $M$, we have
	\[\int_M|\nabla\varphi|^2\geq \int_M (|A_M|^2+\operatorname{Ric}_X(\nu,\nu))\varphi^2\]
	or the Jacobi operator $J=\Delta_M+|A_M|^2+\operatorname{Ric}_N(\nu,\nu)$ is nonnegative. We say that the hypersurface $M$ has finite index if $\sup_i\operatorname{Ind}(\Omega_i)<\infty$ where $\{\Omega_i\}$  are compact subsets of $M$ exhausting $M$. $M$ is stable if $\operatorname{Ind}(\Omega_i)=0$ for any $i$.
	
	In the last few years, there were some exciting progress toward the stable Bernstein problem in Euclidean space. O. Chodosh and C. Li \cite{chodoshliR4} eliminated the volume growth condition in $\mathbb{R}^4$ by using level set method, thus proving the 4-dimensional stable Bernstein theorem. G. Catino, P. Mastrolia and A. Roncoroni \cite{catino} provided a different proof by using the conformal method and the volume comparison under the nonnegative 2-Bakry-Emery-Ricci curvature. Later, O. Chodosh and C. Li \cite{chodoshliR4anisotropic} discovered a new method in $\mathbb{R}^4$ based on the theory of the $\mu$-bubble due to M. Gromov \cite{gromovmububble}. This idea was also successively applied to $\mathbb{R}^5$\cite{chodoshliR5} and $\mathbb{R}^6$\cite{mazet}. Thus the stable Bernstein problem remains open in  $\mathbb{R}^7$.
	
	The stability concept also makes sense for constant mean curvature (CMC) hypersurfaces. Let $M^n\subset X^{n+1}$ be a CMC hypersurface, we say that it is weakly stable if
	\[\int_M |\nabla\varphi|^2\geq \int_M (|A_M|^2+\operatorname{Ric}_X(\nu,\nu))\varphi^2\]
	for any $\varphi\in C_0^\infty(M)$ satisfying
	\[\int_M \varphi=0.\]
	This stems from the second derivative of area functional under the compactly supported variation that preserves the enclosed volume. That $M$ has finite index can defined same as the minimal case above.

	H. Alencar and M. do Carmo \cite{alencardocarmopolynomialgrowth} proved that a complete CMC hypersurface in $\mathbb{R}^{n+1}$ with finite index and with polynomial volume growth must be minimal. Since entire CMC graphs are stable and of polynomial volume growth, their result generalized S. Chern's theorem \cite{cherntheorem}.  The polynomial growth was relaxed to sub-exponential growth later by M. do Carmo and D. Zhou \cite{docarmozhoudetangsubexponential}. Thus M. do Carmo posed the question \cite{docarmoquestionlecturenote}: Is a complete noncompact weakly stable (or generally finite index) constant mean curvature hypersurface in $\mathbb{R}^{n+1}$, where $n\geq 3$, necessarily minimal ? This question had already been answered positively for $n=2$ (see \cite{daSilveira-Stability-of-complete}). Later, the answer was proved to be positive for $n=3,4$ (by M. Elbert, B. Nelli and H. Rosenberg\cite{elbertnellirosenbergpams}, as well as independently by X. Cheng \cite{chengxufiniteindex}, both using a generalized Bonnet-Meyer's argument). Assuming the volume entropy of $M$ is zero, S. Ilias, B. Nelli and M. Soret \cite{zeroentropy} provided a positive answer to M. do Carmo's question in all dimensions. Thus, M. do Carmo's question remains open in higher dimensions.
	
	The hypersurfaces with constant mean curvature one in hyperbolic space behave akin to minimal hypersurfaces in Euclidean space (see \cite{bryant}). Thus, regarding do Carmo's question, it is natural to ask: Does a complete noncompact constant mean curvature hypersurface $\mathbb{H}^{n+1}$, $n\geq 2$, with mean curvature $H\geq 1$ and finite index necessarily have mean curvature $1$ ? 
	
	In this paper, we answer the question positively in hyperbolic space $\mathbb{H}^4$ under extra conditions.  Specifically, we prove
	\begin{theorem}\label{maintheoremintroduction}
		There is no complete, noncompact, constant mean curvature hypersurface $M$ in $\mathbb{H}^4$ with finite-dimensional $H_0^1(M)$, a finite number of ends, and 
		\[H_M>1\ \ \text{and}\ \ \operatorname{Ind}(M)<\infty.\]
	\end{theorem}
	This is equivalent to that complete CMC hypersurfaces in $\mathbb{H}^4$
	satisfying four
	conditions in the theorem are compact. Here $H_0^1(M)$ denotes the first cohomology group with compact support of $M$. In fact, the assumption of finite-dimensional $H_0^1(M)$ and a finite number of ends can be replaced by the finite first $L^2$-Betti number, i.e., $\dim(H_2^1(M))<\infty$ (see Lemma \ref{carron's lemma}). The hypothesis of the mean curvature in the theorem is sharp since the horosphere is a hypersurface of constant mean curvature one and is stable, naturally of finite index. It is worth mentioning that there exist rotational spherical catenoids with finite total curvature and finite index in $\mathbb{H}^4$ with $0\leq H<1$ and they are complete noncompact \cites{docarmorotation, berarddocarmosantosfinitetotalcurvature}. In fact, we prove a more general nonexistence theorem in 4-dimensional Riemannian manifolds (see Theorem \ref{maintheoremlater}). 
	
	Previous to our result, this was proved by A. Da Silveira \cite{daSilveira-Stability-of-complete} in dimension $n+1=3$ without any extra condition. For higher dimensions, there are only partial results. X. Cheng \cite{chengxufiniteindex} proved that any complete finite index hypersurface in hyperbolic spaces $\mathbb{H}^4$ and $\mathbb{H}^5$ with constant mean curvature $H$ satisfying $H^2>\frac{10}{9}$ and $H^2>\frac{7}{4}$, respectively, must be compact. These two numbers were improved to $\frac{64}{63}$ in $\mathbb{H}^4$ and $\frac{175}{148}$ in $\mathbb{H}^5$ by Q. Deng\cite{dengqintao}. As mentioned before, the dimension restriction and strange mean curvature assumption come from the Bonnet-Meyer's type method and appear a lot in the study of the structure of CMC hypersurfaces. In fact, CMC hypersurface $M$ in general ambient Riemannian manifold with following assumption (see \cites{dengqintao,chengxustructure})
	\[\inf_M\operatorname{k-biRic}\geq -\frac{4k^2-n+1}{4((k-1)n+1)}n^2H_M^2\]
	for $k\in (\frac{n-1}{n},\frac{4}{n-1})$ and $n=3,4,5$ usually owns compactness or connectedness at infinity. This method seems to prevent the generalization to higher dimensions and the optimal mean curvature. On the other hand, in higher dimensional hyperbolic space $\mathbb{H}^{n+1}$, S. Ilias, B. Nelli and M. Soret \cite{zeroentropy} solved this question for all $n$ but assuming zero volume entropy of $M.$
	
	Our idea of proof is different and uses a different assumption from zero volume entropy assumption for $\mathbb{H}^4$  . The proof relies on the volume growth estimate based on the warped $\mu$-bubble technique. We outline the proof as follows. 
	\begin{itemize}
		\item Since $M$ has finite index, we examine the components of the complement of a sufficiently large compact subset (e.g., a metric ball of large radius) in $M$. These components are stable manifolds with compact boundary. Furthermore, the stability inequality ensures the validity of a Poincar\'e inequality on the complement. By extending this inequality to the entire manifold (Theorem \ref{poincare inequality}), we deduce a volume inequality.
		\item To estimate the volume growth of geodesic ball, we find a larger set $\tilde{\Omega}$ containing $B_R(p)$ whose boundary consists of $\mu$-bubbles. The area of these bubbles are uniformly controlled if $M$ is a CMC hypersurface  with $H_M>1$ in hyperbolic space, or even more general manifolds with the following $k$-weighted scalar curvature condition (Theorem \ref{bubblediameterestiamte}):
		$$\inf_M \operatorname{k-R_X} +(6k+6)H_M^2\geq 2\epsilon_0$$ for a positive constant $\epsilon_0$ and $k\in[1/2,2]$.
		This part restricts our dimension to $n=4$. However, the number of bubbles may not be uniformly bounded, especially when the manifold is not simply-connected.

		\item To control the number of bubbles, we find an exhaustion $\Omega_R$ of the hypersurface $M$. By using a topological lemma, we know that the number of boundary components of $\Omega_R$ is uniformly bounded. Since we lack of isoperimetric inequality for CMC hypersurfaces, we can not bounded the volume of $\Omega_R$ directly. Instead, we bound the volume of slices by using Bishop-Gromov volume comparison theorem. Eventually, we show that $M$ has finite volume, which is a contradiction.
	\end{itemize}
	The  volume growth of the end is also studied in \cite{chodoshlistrykervolumegrowthproccedings} under nonnegative Ricci curvature condition.

	One interesting consequence of Theorem \ref{maintheoremlater} is the following theorem. 
	
	\begin{theorem}
		\label{theorem2}
		Let $X^4$ be a Riemannian manifold with $\operatorname{Ric}_X \geq 1$. Then there exists no complete noncompact minimal hypersurface $M^3 \subset X$ satisfying both:
		
		\begin{itemize}[leftmargin=2em]
			\item Finite index (including the stable case), and
			\item A finite number of ends and finite-dimensional $H_0^1(M)$.
		\end{itemize}
	\end{theorem}
	
	This theorem can be thought as a noncompact analogue of the well-known result: stable closed minimal hypersurface does not exist in a space with positive Ricci curvature.
	
	\begin{remark}
		The condition of finite topology is important. The following example is provided by Chodosh-Li-Stryker \cite[Appendix A.2]{chodosh-li-stryker}. Consider $M_0=(\mathbb{S}^1\times \mathbb{S}^2)\#(\mathbb{S}^1\times \mathbb{S}^2)$ with a scalar flat metric $h_0.$ The universal cover $\tilde{M}$ of $M_0$ has a positive lower bound $\epsilon$ for its essential spectrum. Consider $X=M_0\times (-\delta,\delta)$ with metric $h_t=h_0+t^2(\operatorname{Ric}_{h_0}-2\epsilon h_0)+dt^2$ for sufficiently small $\delta$ and $t\in (-\delta,\delta)$. Then $\operatorname{Ric}_X\geq \epsilon$ and $\tilde{M}$ is a complete noncompact stable minimal (in fact, totally geodesic) hypersurface  in $X$. However, the topology of $\tilde{M}$ is not finite since it has infinitely many ends. 
	\end{remark}
	\begin{remark}
		The manifold $X$ in last remark is not complete. Chodosh-Li-Strker raised a question that whether completeness of the ambient manifold is important in proving nonexistence of complete noncompact stable minimal hypersurface in space with positive Ricci curvature. We provide a different assumption for this question instead of completeness.
	\end{remark}
	
	We end the introduction with two questions. 
	
	\vskip.1cm
	\noindent\textit{Question.}
	A complete constant mean curvature hypersurface $M$ in $\mathbb{H}^4$ satisfying $H_M \in(1,\frac{64}{63}]$ and $\operatorname{Ind}<\infty$ without zero volume entropy condition or finite topology must be compact ?
	\vskip.1cm
	
	\vskip.1cm
	\noindent\textit{Question.}
	A (weakly) stable complete noncompact constant mean curvature hypersurface in $\mathbb{H}^{n+1}$ with $H=1$ and $n+1=4$ must be horosphere ?
	\vskip.1cm

	The answer to above questions together will solve the generalized stable Bernstein problem in $\mathbb{H}^4$ (see section \ref{last section}). The case $n+1=3$ was solved by A. Da Silveira \cite{daSilveira-Stability-of-complete}.

	\subsection{Acknowledgements}
	We are grateful to Guoyi Xu, Gaoming Wang, Zetian Yan for some helpful discussions. We also thank professors Jingyi Chen and Haizhong Li for their interests and the last section benefits from a discussion with Jingyi Chen. We appreciate the referee for suggestions that greatly improve the paper. The author is supported by NSFC No. 12401058 and the Talent
	Fund of Beijing Jiaotong University No. 2024XKRC008.

	\section{Preliminary}

	Parabolicity and nonparabolicity are important concepts regarding complete noncompact manifolds. It has been thoroughly 
	studied by using harmonic function theory in \cite{li-tam-harmonicfunctiontheory}. We recall some basics for convenience of readers.
	
	\begin{definition}
		A manifold $M$ is parabolic if it does not admit a symmetric, positive Green's function. Otherwise it is nonparabolic.
	\end{definition} 
	
	The notion of parabolicity can be localized at an end of the manifold. Let $\Omega$ be a compact subset of $M$, an end $E$ with respect to $\Omega$ is simply an unbounded component of $M\setminus \Omega$. An end $E$ is parabolic if it does not admit a positive harmonic function $f$ satisfying
	\[f|_{\partial E}=1\ \ \ \text{and}\ \ \ f|_{E^\circ}<1.\]
	Otherwise, it is nonparabolic.
	A complete noncompact $M$ is nonparabolic if and only if $M$ has a nonparabolic end. It is also possible for a nonparabolic manifold to have many parabolic ends.  The compact set $\Omega$ is usually set to be a geodesic ball with large radius $R$ and centered at point $p$. Hereafter, without ambiguity, we simply write $B_R$ instead of $B_R(p).$ 
	
	
	Let $E_0$ be one component of $M\setminus B_{R_0}$ and $L$ be a positive number. One then considers the sequence of harmonic functions defined on $E_k=E_0\cap B_{kL+R_0}$ satisfying
	\[\begin{cases}
		f_k=0 &\text{on} \  E_{0}\cap \partial B_{kL+R_0}
		\\
		f_k=1 & \text{on} \ \partial E_0
	\end{cases}\]
	Notice that $E_0=\cup_{i=1}^{\infty}E_k$. We also often use $\{E_k\}_{k=1}^{\infty}$ to denote the end $E_0$. An end $\{E_k\}_{k=1}^{\infty}$ with respect to $B_{R_0}$ is nonparabolic if and only if $f_k$ converges uniformly and locally to a nonconstant function $f$ defined on $E_0$ as $k\rightarrow \infty$. $f\equiv 1$ when it is parabolic.
	
	We first prove 
	\begin{lemma}\label{noparabolicends}
		Let $M$  be a complete CMC hypersurface in $\mathbb{H}^{n+1}$ with finite index and the scalar (averaged) mean curvature satisfies $H_M>1$. Then $M$ has no parabolic ends.
	\end{lemma}
	\begin{proof}
		Suppose $M$ is noncompact, otherwise we are done. Since $M$ has finite index, it is well-known that outside a compact subset $\Omega$ of $M$, we have
		\begin{align*}
			\int_{M\setminus \Omega}|\nabla \varphi|^2&\geq \int_{M\setminus \Omega}(|\mathring{A}|^2+nH_M^2-n)\varphi^2\\
			&\geq n(H_M^2-1)\int_{M\setminus \Omega}\varphi^2
		\end{align*}
		for $\varphi\in C_c^\infty(M\setminus \Omega)$. Choose $p\in M$  and $R_0>0$ such that $B_p(R_0)$ contains $\Omega$. Let $\{E_k\}_{k=1}^{\infty}$ be an end with respect to $B_p(R_0)$. 
		
		Fix $k_0>1$. Let $\varphi$ be a cut-off function such that
		\[\varphi=1 \ \text{on}\ E_{k_0+\ell}\setminus E_{k_0},\ \ \varphi=0\ \text{on}\ \partial E_0, \ \text{and}\ |\nabla\varphi|\leq C\]
		where $C$ is independent of $\ell$. Plug $\varphi f_k$ into above inequality, since $f_k$ is harmonic, we obtain by integration of by parts
		\begin{align*}
			(nH_M^2-n)\int_{E_0} (\varphi f_{k})^2 &\leq \int_{E_0}|\nabla(\varphi f_{k})|^2\\
			&=\int_{E_0} |\nabla \varphi|^2 f_{k}^2.
		\end{align*}
		If $\{E_{k}\}_{k=1}^{\infty}$ is parabolic, then $f_{k}\rightarrow 1$ as $k\rightarrow \infty$. Above inequality turns
		\[(nH_M^2-n)|E_{k_0+\ell}\setminus E_{k_0}|\leq C|E_{k_0}| \]
		for any $\ell>0$.  This implies that volume of the end $\{E_k\}$ is finite.
		
		On the other hand, the volume of each end (parabolic or nonparabolic) of a CMC hypersurface in a Riemannian manifold with bounded geometry is infinite. Hyperbolic space has bounded geometry. Thus all ends of $M$ are nonparabolic if there exist ends, completing the proof.
	\end{proof}
	
	As we can see from the proof, a more general nonexistence theorem holds in Riemannian manifolds with bounded geometry.
	
	\begin{lemma}\label{noparabolicendsingeneralmanifold}
		Let $M$ be a complete CMC hypersurface in $(X,g)$ with finite index. Assume that $nH_M^2+\inf_M \operatorname{Ric}_X>0$ and $X$ has bounded geometry. Then $M$ has no parabolic ends.
	\end{lemma}
	
	Note that a manifold being nonparabolic is weaker than saying that it has no parabolic ends.

	\section{$\mu$-bubble, diameter estimate and volume bound}\label{volumegap}
	In this section, we observe that the theory of the warped $\mu$-bubble used in \cite{chodosh-li-stryker} can yield volume bound for ends of CMC hypersurface in Riemannian manifold with possibly negative  curvature.
	
	Let $M^n$ be a hypersurface immersed in $X^{n+1}$ with the unit normal vector field $\bar{\nu}$ and $\Sigma^{n-1}$ be a hypersurface in $M^n$ the unit normal vector field $\nu$. Let $R_X, R_\Sigma$ be the scalar curvatures of $X, \Sigma$ respectively. Let $\operatorname{Ric}_X, \operatorname{Ric}_M$ denote the Ricci curvature of $X, M$ respectively. Let $A_M$ and $A_\Sigma$ denote the second fundamental form of $M$ in $X$ and $\Sigma$ in $M$ respectively. Suppose $H_M$ and $H_\Sigma$ are the scalar mean curvature of $M$ in $X$ and $\Sigma$ in $M$ respectively, i.e., the average of the trace of $A_M$ and  $A_\Sigma$. 
	
	\begin{definition}
		Given an $(n+1)$-dimensional Riemannian manifold $(X^{n+1},g)$, where $n\geq 2$, and one unit tangent vector $V\in T_pX$ at $p$ , the $k$-weighted scalar curvature in the direction $V$ at $p$ is defined by
		\[\operatorname{k-R_X}(V):=\operatorname{Ric}(e_1,e_1)+\cdots+\operatorname{Ric}(e_{n},e_n)+(2k-1)\operatorname{Ric}(V,V)\]
		where $\{e_1,\cdots,e_n,V\}$ forms an orthonormal basis of $T_pM.$
	\end{definition} 
	When $k=1$, $\operatorname{k-R_X}$ is just the Scalar curvature of $X.$ It can also be written as $\operatorname{k-R_X}(V)=R_X+2(k-1)\operatorname{Ric}(V,V).$
	
	\begin{proposition}\label{curvature relation}
		Under above setting, there hold
		\[R_X-R_\Sigma+n^2H_M^2+(n-1)^2H_\Sigma^2=2\operatorname{Ric}_X(\bar{\nu},\bar{\nu})+2\operatorname{Ric}_M(\nu,\nu)+|A_\Sigma|^2+|A_M|^2\]
		and for any $k\geq \frac{1}{2}$,
		\begin{align*}
			k\operatorname{Ric}_X(\bar{\nu},\bar{\nu})+k|A_M|^2+|A_\Sigma|^2&+\operatorname{Ric}_M(\nu,\nu)-\frac{1}{2}(n-1)^2H_\Sigma^2\\&\geq \frac{1}{2}\inf_M \operatorname{k-R_X} +\frac{n(2k+n-1)}{2}H_M^2-\frac{1}{2}R_\Sigma.
		\end{align*}
	\end{proposition}
	\begin{proof}
		It follows from the definition directly. Choose a local orthonormal frame $\{e_1,\cdots, e_{n+1}\}$ at a point $p$ in $X$ such that $\{e_1,\cdots, e_{n-1}\}$ is a local orthonormal frame for $\Sigma$ and $\nu=e_{n}, \bar{\nu}=e_{n+1}$. Let $h^M_{ij}$ and $h^\Sigma_{ij}$ be the scalar second fundamental form of $A_M$ and $A_\Sigma$ respectively. Let $R^X,R^M, R^\Sigma$ be the curvature tensor of $X, M, \Sigma$ respectively. Thus by the Gauss equation repeatedly,
		\begin{align*}
			R_X-2Ric_X(\bar{\nu},\bar{\nu})&=2\sum_{1\leq i<j\leq n} R^X_{ijij}\\
			&= 2\sum_{1\leq i<j\leq n} (R^M_{ijij}-h^M_{ii}h^M_{jj}+(h^M_{ij})^2)\\
			&=2\sum_{1\leq i<j\leq n-1}R^M_{ijij}+2\sum_{i=1}^{n-1} R^M_{inin}+|A_M|^2-n^2H_M^2\\
			&=2\sum_{1\leq i<j\leq n-1}R^M_{ijij}+2\operatorname{Ric}_M(\nu,\nu)+|A_M|^2-n^2H_M^2\\
			&=R_\Sigma+|A_\Sigma|^2-(n-1)^2H_\Sigma^2+2\operatorname{Ric}_M(\nu,\nu)+|A_M|^2-n^2H_M^2.
		\end{align*}
		Rearranging terms completing the first equation. Furthermore, we have
		\begin{align*}
			\operatorname{Ric}_M(\nu,\nu)+|A_\Sigma|^2&=\frac{1}{2}(R_X-R_\Sigma+n^2H_M^2+(n-1)^2H_\Sigma^2+|A_\Sigma|^2-|A_M|^2)\\&-\operatorname{Ric}_X(\bar{\nu},\bar{\nu}))\\
			&\geq \frac{1}{2}(R_X-2\operatorname{Ric}_X(\bar{\nu},\bar{\nu}))\\
			&+\frac{1}{2}((n-1)^2H_\Sigma^2+n^2H_M^2-|A_M|^2-R_\Sigma)
		\end{align*}
		Then by definition,
		\begin{align*}
			k\operatorname{Ric}_X&(\bar{\nu},\bar{\nu})+k|A_M|^2+|A_\Sigma|^2+\operatorname{Ric}_M(\nu,\nu)\\&\geq (\frac{1}{2}R_X+(k-1)\operatorname{Ric}_X(\bar{\nu},\bar{\nu}))+(k-\frac{1}{2})|A_M|^2\\
			&+\frac{1}{2}n^2H_M^2+\frac{1}{2}(n-1)^2H_\Sigma^2-\frac{1}{2}R_\Sigma\\
			& \geq \frac{1}{2}\inf_M \operatorname{k-R_X} +\left(n(k-\frac{1}{2})+\frac{1}{2}n^2\right)H_M^2+\frac{1}{2}(n-1)^2H_\Sigma^2-\frac{1}{2}R_\Sigma.
		\end{align*}
		This completes the proof of the second equation.
	\end{proof}
	
	\subsection{Closed $\mu$-bubble} We first prove a diameter bound for the warped $\mu$-bubble in ends of CMC hypersurface with finite index. This uses the theory of $\mu$-bubble introduced by M. Gromov and applied to stable minimal hypersurfaces in 4-manifolds by O. Chodosh, C. Li and D. Stryker \cite{chodosh-li-stryker}. 
	\begin{theorem}
		\label{bubblediameterestiamte}
		Let $(X^4,g)$ be a complete 4-dimensional Riemannian manifold and $M^3\subset X^4$ be a stable two-sided CMC hypersurface with  the compact boundary component $\partial M$ and of the mean curvature $H_M$. Let $$\inf_M \operatorname{k-R_X} +(6k+6)H_M^2\geq 2\epsilon_0$$ for a positive constant $\epsilon_0$ and some $k\in[1/2,2]$. Suppose that there exists $p\in M$ such that $d_M(p,\partial M)\geq L/2$ for a positive constant $L>8\pi/\sqrt{\epsilon_0}$. Then there exists a relatively open set $\tilde{\Omega} \subset B^M_{L/2}(\partial M)$ and a smooth surface $\Sigma$ (possibly disconnected) such that $\Sigma\cup \partial M=\partial \tilde{\Omega}$ and the diameter of each component of $\Sigma$ is at most $\frac{2\pi}{\sqrt{\epsilon_0}}$.  Moreover, each component is a sphere and the area of each component is bounded by $8\pi/\epsilon_0.$
		
	\end{theorem}
	
	\begin{proof}
		This is from the technique of $\mu-$bubble that has been studied and applied in many settings.  Since $M$ is (strongly) stable, we have that for any compactly supported function $\varphi$ on $M$,
		\[\int_M |\nabla \varphi|^2-(\operatorname{Ric}_X(\bar{\nu},\bar{\nu})+|A_M|^2)\varphi^2\geq 0.\]
		It follows that there exists a positive function $u\in C^\infty(M\setminus \partial M)$ such that
		\begin{equation}\label{positive u on m}
			\Delta_M u+(\operatorname{Ric}_X(\bar{\nu},\bar{\nu})+|A_M|^2)u\leq 0.
		\end{equation} 
		When $M$ is compact, the above conclusion is trivial. 
		Otherwise, it can be obtained by following D. Fischer-Colbrie and R. Schoen's paper \cite{Fischer-Colbrie-Schoen-The-structure-of-complete-stable}. 
		The rest of the proof uses the second variation of the warped $\mu$-bubble and diameter estimate (see \cite{chodosh-li-stryker}). 
		
		Let $\rho_0$ be a smoothing of the distance function $d(x,\partial M)$ such that $|\nabla \rho_0|\leq 2$. For a fixed small $\epsilon>0$, define
		\[h(x)=-\frac{4\pi}{L/2-2\epsilon}\tan(\rho)\]
		where $$\rho=\frac{(\rho_0-\epsilon)\pi}{L/2-2\epsilon}-\frac{\pi}{2}.$$
		
		We use functions $u$ and $h$ to construct $\mu$-bubble $\Sigma$ (possibly disconnected) and the domain $\tilde{\Omega}$ enclosed by $\Sigma$ and $\partial M$. Let $\Omega_0=\{x\in M: \epsilon< \rho_0< \frac{L}{4}\}$ and $\Omega_1=\{x\in M: \epsilon<\rho_0\leq \frac{L}{2}-\epsilon \}$. For any $k\in [1/2,2]$, consider
		\[\mathcal{A}(\Omega)=\int_{\partial^{*} \Omega} u^k-\int_{\Omega_1}(\chi_\Omega-\chi_{\Omega_0})u^kh\]
		for Caccioppoli sets $\Omega$ such that $\Omega\Delta\Omega_0$ is compactly supported  in $\Omega_1^\circ.$ There exists a smooth minimizer, still denoted by $\Omega$, of $\mathcal{A}(\cdot)$ since $u^k$ is positive. For simplicity, let $\Sigma$ be one of the component of $\partial\Omega$ contained in $M^\circ.$
		
		The critical surface $\Sigma$ satisfies
		\[3H_\Sigma=-k\frac{D_\nu u}{u}+h.\]
		And the nonnegative second derivative of $\mathcal{A}(\Omega)$ at the critical point yields that for any function $\varphi\in C^\infty(\Sigma),$
		\begin{equation}   \label{secondderivativeofmububble}  
			\begin{aligned}
				\int_\Sigma u^k|\nabla^\Sigma \varphi|^2-ku^{k-1}\varphi^2&\Delta_\Sigma u \geq \int_\Sigma -ku^{k-1}\varphi^2\Delta_M u+(|A_\Sigma|^2+\operatorname{Ric}_M(\nu,\nu))u^k\varphi^2\\
				&+\int_\Sigma \varphi^2u^k D_\nu h+k\varphi^2u^{k-1}hD_\nu u-k(k-1)u^{k-2}\varphi^2(D_\nu u)^2\\
				&=\int_\Sigma (-k\frac{\Delta_M u}{u}+|A_\Sigma|^2+\operatorname{Ric}_M(\nu,\nu)-\frac{9}{2}H_\Sigma^2)u^k\varphi^2\\
				&+\int_\Sigma \frac{1}{2}(h^2+2D_\nu u)u^k\varphi^2+(\frac{1}{2}k^2-k(k-1))u^{k-2}\varphi^2(D_\nu u)^2\\
				&\geq \int_\Sigma (-k\frac{\Delta_M u}{u}+|A_\Sigma|^2+\operatorname{Ric}_M(\nu,\nu)-\frac{9}{2}H_\Sigma^2)u^k\varphi^2\\
				&+\int_\Sigma \frac{1}{2}(h^2+2D_\nu h)u^k\varphi^2.
			\end{aligned}
		\end{equation}
		
		Let $\varphi=u^{-k/2}\phi$, plugging it in the left-hand side of the above equation, by divergence theorem and Cauchy-Schwarz inequality we obtain
		\begin{align*}
			\int_\Sigma u^k|\nabla^\Sigma \varphi|^2-ku^{k-1}\varphi^2\Delta_\Sigma u&\leq \int_\Sigma (1+\frac{k}{2\varepsilon})|\nabla^\Sigma \phi|^2+(k(\frac{\varepsilon}{2}-1)+\frac{k^2}{4})u^{-2}\phi^2|\nabla^\Sigma u|^2\\
			&\leq \frac{4}{4-k}\int_\Sigma |\nabla^\Sigma \phi|^2
		\end{align*}
		by choosing $\varepsilon=2-\frac{k}{2}.$
		
		Hence,
		\begin{align*}
			\frac{4}{4-k}\int_\Sigma |\nabla^\Sigma \phi|^2&\geq \int_\Sigma (-k\frac{\Delta_M u}{u}+|A_\Sigma|^2+\operatorname{Ric}_M(\nu,\nu)-\frac{9}{2}H_\Sigma^2-\frac{\epsilon_0}{2})\phi^2\\
			&+\int_\Sigma \frac{1}{2}(\epsilon_0+h^2+2D_\nu h)\phi^2.
		\end{align*}
		The last term is nonnegative if $L$ satisfies
		\[\epsilon_0+h^2+2D_\nu h\geq \epsilon_0-\left(\frac{4\pi}{L/2-2\epsilon}\right)^2\geq 0,\]
		that is
		\[L\geq 2(\frac{4\pi}{\sqrt{\epsilon_0}}+2\epsilon).\]
		
		Using equation (\ref{positive u on m}) and Proposition \ref{curvature relation}, we have
		
		\begin{align}\label{stabilityinequalityforSigma}
			0\leq & \int_\Sigma |\nabla^\Sigma \phi|^2-\frac{4-k}{4}(|A_\Sigma|^2+\operatorname{Ric}_M(\nu,\nu)-\frac{9H_\Sigma^2}{2}+k\operatorname{Ric}_X(\bar{\nu},\bar{\nu})+k|A_M|^2-\frac{\epsilon_0}{2})\phi^2\\
			\leq & \int_\Sigma |\nabla^\Sigma \phi|^2-\frac{4-k}{4}(\frac{1}{2}\inf_M \operatorname{k-R_X} +(3k+3)H_M^2-\frac{\epsilon_0}{2}-\frac{R_\Sigma}{2})\phi^2\nonumber\\
			\leq &\int_\Sigma |\nabla^\Sigma \phi|^2-\frac{4-k}{4}(\frac{\epsilon_0}{2}-\frac{R_\Sigma}{2})\phi^2\nonumber 
		\end{align}
		for any $\phi\in C^\infty(\Sigma)$.
		Thus there exists a positive function $v\in C^\infty(\Sigma)$ such that 
		\[\Delta v+\frac{4-k}{4}(\frac{\epsilon_0}{2}-K_\Sigma)v\leq 0.\]
		
		We can not directly apply the diameter estimate in  \cite[Lemma 16]{chodosh-li-bubble} due to the coefficient in front of $K_\Sigma.$ However, we can still follow the idea of \cite[Lemma 16]{chodosh-li-bubble} to show
		\[\operatorname{diam}(\Sigma)\leq \frac{2\pi}{\sqrt{\epsilon_0}}.\]
		
		In fact, assume that there exists  $p,q\in M$ with $d(p,q)>\frac{2\sqrt{2}\pi}{\sqrt{\epsilon_0}}$. Following \cite[Lemma 16]{chodosh-li-bubble}, there is $h>0$ such that $\frac{\epsilon_0}{2}+\frac{1}{2}h^2+D^\Sigma_\nu h>0$. Let $b=\frac{4}{4-k}$. Consider
		\[\mathcal{A}(K)=\int_\gamma v^b-\int_{K}(\chi_K-\chi_{K_0})hv^b\]
		where $K_0$ is the reference Caccioppoli set $B_\epsilon(p)\subset K_0\subset \Sigma\setminus B_\epsilon(q).$
		The minimizing closed $\mu$-curve $\gamma$ exists in $\Sigma$. By conducting similar calculation as in \eqref{secondderivativeofmububble} to $\gamma$ we have 
		\begin{align*}
			\int_\gamma v^b|\nabla^\gamma\varphi|-bv^{b-1}\varphi^2\Delta_\gamma v&\geq \int_\gamma(-b\frac{\Delta_\Sigma v}{v}+\frac{1}{2}\kappa_\gamma^2+K_\Sigma)v^b\varphi^2+\frac{b}{2}(2-b)u^{b-2}\varphi^2(D_\nu v)^2\\&+\int_\gamma(\frac{1}{2}h^2+D_\nu h)v^b\varphi^2\\
			&\geq \int_\gamma(\frac{\epsilon_0}{2}+\frac{1}{2}h^2+D_\nu h)v^b\varphi^2+\int_\gamma \frac{b}{2}(2-b)u^{b-2}\varphi^2(D_\nu v)^2.
		\end{align*}
		Plugging $\varphi=v^{\frac{2}{k-4}}$, we conclude that
		\begin{align*}
			\int_\gamma \frac{4(k-3)}{(4-k)^2}v^{-2}|\nabla^\gamma v|^2\geq \int_\gamma (\frac{\epsilon_0}{2}+\frac{1}{2}h^2+D_\nu h)+\int_\gamma \frac{8-4k}{(4-k)^2}v^{-2}(D^\Sigma_\nu v)^2
		\end{align*}
		leading to a contradiction for any $k\in [1/2,2]$.
		
		Lastly, by Gauss-Bonnet formula we have
		\[\frac{\epsilon_0}{2}|\Sigma|\leq \int_\Sigma K_\Sigma=2\pi(2-2g),\]
		so the genus is zero and $|\Sigma|\leq 8\pi/\epsilon_0.$
		
		Let $\tilde{\Omega}=\Omega \cup \{\rho_0\leq \epsilon\}.$  We complete the proof.
		
	\end{proof}

	\subsection{A volume inequality and topological lemmas}
	In this section, we prove a volume inequality and prepare some topological lemmas that will be essential in the proof of the main theorem.
	
	\begin{lemma}\label{poincare inequality}
		Let $(X^4,g)$ be a complete 4-dimensional Riemannian manifold with bounded geometry and let $M$ be a noncompact CMC hypersurface in $X$ with finite index. If $3H_M^2+\inf_M\operatorname{Ric}_X\geq\delta$ for some constant $\delta>0$, then there exists a constant $C=C(\delta, M)>0$  such that
		\begin{equation}\label{poincare inequality eq}
			\int_M\varphi^2\leq C\int_M  |\nabla\varphi|^2
		\end{equation}
		for all $\varphi\in C^1_c(M).$
	\end{lemma}
	\begin{proof}
		Since $M$ has finite index, there exists a compact subset $\Omega\subset M$ such that $M\setminus \Omega$ is stable. This means that for any $\varphi\in C_c^1(M\setminus \Omega)$, we have the stability inequality:
		\begin{align*}
			\int_M|\nabla \varphi|^2&\geq \int_M (|A_M|^2+\operatorname{Ric}_X(\nu,\nu))\varphi^2\\
			&\geq \int_M(3H_M^2+\operatorname{Ric}_X(\nu,\nu))\varphi^2\\
			&\geq \delta \int_M \varphi^2.
		\end{align*}
		
		We now use the Poincar\'e inequality and the nonparabolicity of $M$ to extend above inequality to $M.$ Let $\tilde{\Omega}$ be a slightly larger compact set and let $\theta$ be a smooth cutoff function  that is $1$ on $\Omega$ and $0$ in $M\setminus \tilde{\Omega}$, moreover, we can assume that $|\nabla\theta|\leq 1$. Let $\varphi\in C_c^1(M)$, note that $(1-\theta)\varphi$ is compactly supported in $M\setminus \Omega$. Denote the constant $\int_{\tilde{\Omega}}\varphi/|\tilde{\Omega}|$ by $\tilde{\varphi}$. Then
		\begin{align*}
			\|\varphi\|_{L^2(M)}&\leq \|(1-\theta)\varphi\|_{L^2(M)}+\|\theta(\varphi-\tilde{\varphi})\|_{L^2(M)}+\|\theta\tilde{\varphi}\|_{L^2(M)}\\
			&\leq \frac{1}{\delta}\|\nabla ((1-\theta)\varphi)\|_{L^2(M\setminus \Omega)} +\|\varphi-\tilde{\varphi}\|_{L^2(\tilde{\Omega})}+\|\tilde{\varphi}\|_{L^2(\tilde{\Omega})}.
		\end{align*}
		The first term becomes
		\begin{align*}
			\|\nabla ((1-\theta)\varphi)\|_{L^2(M\setminus \Omega)} &\leq \|\varphi\|_{L^2(\tilde{\Omega})}+\|\nabla\varphi\|_{L^2(M)}\\
			&\leq \|\varphi-\tilde{\varphi}\|_{L^2(\tilde{\Omega})}+\|\tilde{\varphi}\|_{L^2(\tilde{\Omega})}+\|\nabla\varphi\|_{L^2(M)}.
		\end{align*}
		By Poincar\'e inequality on $\tilde{\Omega}$, we have
		\[\|\varphi-\tilde{\varphi}\|_{L^2(\tilde{\Omega})}\leq C\|\nabla \varphi\|_{L^2(\tilde{\Omega})}.\]
		Thus we obtain
		\[\|\varphi\|_{L^2(M)}\leq C(\|\nabla \varphi\|_{L^2(M)}+\|\tilde{\varphi}\|_{L^2(\tilde{\Omega})})\]
		where $C$ is a constant depending on $\delta$, $\Omega$ and $\tilde{\Omega}.$
		
		We now estimate $\|\tilde{\varphi}\|_{L^2(\tilde{\Omega})}$. First we have that
		\[\|\tilde{\varphi}\|_{L^2(\tilde{\Omega})}\leq C\left|\int_{\tilde{\Omega}}\varphi\right|.\]
		Notice that $\varphi$ is not necessarily compactly supported in $\tilde{\Omega}$. By Lemma \ref{noparabolicendsingeneralmanifold}, we know that $M$ is nonparabolic. Then it follows from the equivalent definitions of nonparabolicity (see \cite[Definition 2.13]{Carron-L2-harmonic}) that there exists a constant $C>0$ such that
		\[\left|\int_{\tilde{\Omega}}\varphi\right|\leq C\left(\int_M|\nabla\varphi|^2\right)^{\frac{1}{2}}.\]
		
		In conclusion, we have
		\[\|\varphi\|_{L^2(M)}\leq C \|\nabla \varphi\|_{L^2(M)}\]
		completing the proof of the lemma.
		
	\end{proof}
	
	Thus we have the following volume inequality.
	\begin{corollary}\label{areavolumeinequality}
		Let $(X^4,g)$ be a complete 4-dimensional Riemannian manifold with bounded geometry and let $M$ be a noncompact CMC hypersurface in $X$ with finite index. Assume that $3H_M^2+\inf_M\operatorname{Ric}_X\geq\delta$ for a positive constant $\delta$, then there exists a constant $C>0$ such that
		\[ |\Omega|\leq \frac{C}{\tau^2}|\Omega_\tau\setminus \Omega|\]
		for any compact subset $\Omega\subset M$ and $\tau>0$. Here, $\Omega_\tau=\{x\in M:\operatorname{dist}_M(x,\Omega)\leq \tau\}.$
	\end{corollary}
	
	\begin{proof}
		The result follows by plugging the following function into \eqref{poincare inequality eq}.
		\[f(x)=\begin{cases}
			1, & x\in \Omega\\
			\frac{\tau-d_M(x,\partial\Omega)}{\tau}, & x\in \Omega_\tau\setminus \Omega\\
			0, & \text{otherwise}.
		\end{cases}\]
	\end{proof}
	
	\begin{lemma}\label{curvature estimate}
		Let $(X^4,g)$ be a complete $4$-dimensional Riemannian manifold with bounded geometry and let $\Omega$ be a stable compact CMC hypersurface with compact boundary and the mean curvature $H$ in $X.$ Then,
		\[\sup_{x\in \Omega}|A_\Omega|(x)\min\{1,d_\Omega(x,\partial \Omega)\}\leq C\]
		for a constant $C=C(X,g,H)$. In particular, for any noncompact complete stable CMC hypersurface $E_0$ with compact boundary in $X,$ the second fundamental form $|A_{E_0}|$ is uniformly bounded.
	\end{lemma}
	\begin{proof}
		When $M$ is complete ($\partial\Omega=\emptyset$), this has been shown by H. Rosenberg, R. Souam and E. Toubiana \cite{rosenbergcurvatureestimate} in dimension $2$. If $M$ is $3$-dimensional minimal immersion, this was proven by O. Chodosh, C. Li and D. Stryker \cite{chodosh-li-stryker}. The proof of the claim follows in a similar manner as in \cite{chodosh-li-stryker} by using the standard point-picking argument. If the claim is not true, there exists a sequence of CMC immersions $\Omega_i$ with mean curvature $H$ and a sequence of points $p_i\in\Omega_i\setminus \partial\Omega_i$ such that $|A_{\Omega_i}|(p_i)\min\{1,d_{\Omega_i}(p_i,\partial\Omega_i)\}\rightarrow \infty$ as $i\rightarrow \infty$. This means that $|A_{\Omega_i}|(p_i)\rightarrow \infty$. The next step is to dilate the metric by $|A_{\Omega_i}|(p_i)$ and study the convergence of $\Omega_i$. Note that $|A|$ at $p_i$ under the new metric is one. The rest of proof follows as in \cite{chodosh-li-stryker}. The only difference in our case is that we are dealing with CMC hypersurface with nonzero mean curvature. However, notice that the dilations by $|A_{\Omega_i}|$ of CMC sequence will be complete minimal hypersurface. Hence, again we can use stable Bernstein theorem for minimal hypersurfaces in $\mathbb{R}^4$ to obtain contradiction in the point-picking argument. 
	\end{proof}

	The following topological result provides an alternative assumption for the main theorem.
	\begin{lemma}{\cite[Proposition 2.14]{Carron-L2-harmonic}}\label{carron's lemma}
		Let $M$ be a complete noncompact Riemannian manifold. Assume that all ends of $M$ are nonparabolic,
		then
		\[{0}\rightarrow H_0^1(M)\rightarrow H_2^1(M)\]
		where $H_0^1(M)$ denotes the first cohomology group with compact support of $M$ and $H_2^1(M)$ denotes the space of $L^2$-harmonic one-form on $M.$ Moreover, $$\operatorname{dim}(H_2^1( M))\geq \text{the number of ends of} \ M.$$
	\end{lemma}

	If we assume that $\operatorname{dim}(H_2^1( M))$ is finite, then $M$ has finite number of ends and $\operatorname{dim}(H_0^1(M))$ is finite. Hereafter, we denote the number of ends of $M$ by $\ell.$
	The following result will be crucial in the main proof.
	\begin{lemma}[\cite{chodoshliR4}]\label{boundary}
		Suppose that $M$ has $\ell$ ends and $H^{1}_{0}(M;\mathbb{R})$ has finite dimension.
		Consider an exhaustion $\Omega_{1}\subset \Omega_{2}\subset\cdots\subset M$ by pre-compact regions with smooth boundary. If each component of $M\backslash\Omega_{i}$ is unbounded. Then for $i$ sufficiently large, $\partial \Omega_{i}$ has $\ell$ components. 
	\end{lemma}

	\section{Proof of main results}

	In this section we prove the main theorem. In fact, we prove a more general theorem. 
	\begin{theorem}\label{maintheoremlater}
		A complete constant mean curvature hypersurface $M$ in $(X^4,g)$ with the bounded geometry satisfying 
		
		\[3H_M^2+\inf_M\operatorname{Ric}_X\geq\delta\]
		for a positive constant $\delta $ and
		$$\operatorname{Ind}(M)<\infty, \ \ \operatorname{dim}(H_0^1(M))<\infty,\ \ \#\{\text{ends of $M$}\}<\infty$$
		must be compact.
	\end{theorem}

	\begin{proof}
		Suppose that $M$ is noncompact, otherwise there is no need to prove. We can also assume that $M$ is stable outside of a large subset $\Omega$. Fix a point $p\in \Omega$ and consider the intrinsic ball $B_R(p)$ for sufficiently large $R$ such that $\Omega\subset B_R(p)$. Since the number of ends of $M$ is finite, we denote it by $\ell$. Then by choosing a larger $R$ if necessary, there are $\ell$ number of unbounded connected components of $M\setminus B_R$. We denote them by $\cup_{i=1}^\ell E_i$. Furthermore, we denote   the union of bounded components in $M\setminus B_R(p)$ by $S$.
		
		Since $3H_M^2+\inf_M \operatorname{Ric}\geq \delta$ implies the curvature assumption in Theorem \ref{bubblediameterestiamte} for some $k\in[1/2,2]$.  Thus we can apply Theorem \ref{bubblediameterestiamte} to $E_i$. In fact, it implies that $12H_M^2+\inf_M R_X\geq 4\delta$ and by the definition of $\operatorname{k-R_X}$, we have that 
		\[\inf_M \operatorname{k-R_X}+(6k+6)H_M^2\geq 2(k+1)\delta.\]
		Then for each fixed $i$, we have a domain $\tilde{\Omega}_i$ such that $\partial\tilde{\Omega}_i=\partial E_i\cup (\cup_{j=1}^{\ell_i}\Sigma_{ij})$ where $\Sigma_{ij}$ are closed spheres with diameter uniformly bounded from above by a constant $C$ depending on $\delta$. Moreover, we know that for each $i,$ all $\Sigma_{ij}$ are contained in the $L/2$-neighborhood of $\partial E_i$ in $E_i$. Denote $\tilde{S}_i$ the union of bounded components of $E_i\setminus \tilde{\Omega}_i.$ Without lost of generality, by reordering, we denote $\cup_{j=1}^{t_i}\Sigma_{ij}$ the subset of  $\cup_{j=1}^{\ell_i}\Sigma_{ij}$ that are the boundary of the unbounded component of $E_i\setminus \tilde{\Omega}_i$. Note that $\cup_{j=1}^{\ell_i}\Sigma_{ij}=(\cup_{j=1}^{t_i}\Sigma_{ij})\cup \partial \tilde{S}_i$.
		
		We now consider the following exhaustion of $M$ in terms of $R$:
		\[\Omega_R=B_R(p)\cup S \cup (\cup_{i=1}^{\ell}\tilde{\Omega}_i)\cup (\cup_{i=1}^{\ell}\tilde{S}_i).\]
		It is not hard to see that each component in $M\setminus \Omega_R$ is unbounded. By the construction, we also know that $\partial\Omega_R=\cup_{i=1}^{\ell}\cup_{j=1}^{t_i}\Sigma_{ij}$. Then it follows from Lemma \ref{boundary}
		that for $R$ sufficiently large,
		\[\#\{\text{components of} \ \partial\Omega_R\}= \sum_{i=1}^{\ell} t_i=t_1+\cdots+t_\ell=\ell.\]
		That is, $t_i=1$ for any $i$. Thus we can simply denote the boundary components of $\partial\Omega_R$ by $\Sigma_1,\cdots,\Sigma_\ell$, each corresponding to the end $E_i$. The diameter and area of $\Sigma_i$ are uniformly bounded from above.
		
		Without the isoperimetric inequality, we can not directly estimate the volume of $\Omega_R$. Nevertheless, in the following we show that the volume of $\Omega_R$ can be controlled. Consider the set $(\Omega_R)_{L/2}$ defined in Corollary \ref{areavolumeinequality} and  denote $(\Omega_R)_{L/2}\setminus \Omega_R$ by $K=\cup_{i=1}^\ell K_i$. We claim that the diameter of each $K_i$ is bounded by a universal constant $C.$ Indeed, consider any two points $x,y\in K_i$ and connect them to $\Sigma_{i}$ by minimizing geodesic segments. Then,
		by triangle inequality and diameter estimate of bubbles in Theorem \ref{bubblediameterestiamte}, we know that diameter of $K_i$ satisfies
		\[\operatorname{diam}_{K_i}(x,y)\leq L+\frac{2\pi}{\sqrt{\epsilon_0}}.\]
		The Gauss equation yields
		\[\operatorname{Ric}_M(e_1,e_1)=\sum_{i=1}^3R_X(e_1,e_i,e_1,e_i)+3h_{11}H-\sum_{i=1}^3h_{1i}^2.\]
		Since $X$ has bounded geometry, the curvature of $X$ is uniformly bounded. Moreover, Theorem \ref{curvature estimate} says that $K_i$ has uniform bounded second fundamental form. Then the Ricci curvature of $M$ is uniformly bounded from below. Hence, by Bishop-Gromov volume
		comparison theorem and bounded diameter of $K_i$, the volume of $K=\cup_{i=1}^\ell K_i$ has a uniform upper bound, denoted by $C.$

		Hence, by Corollary \ref{areavolumeinequality}, we obtain 
		\begin{align*}
			|B_R(p)|&\leq |\Omega_R|\\
			&\leq \frac{4C}{L^2}| (\Omega_{R})_{L/2}\setminus \Omega_R|\\
			&\leq C|K|\\
			&\leq C.
		\end{align*}
		The constant $C$ differs from line to line, but it is a uniform constant independent of $R$.
		This means that $M$ has finite volume, which is a contradiction as $M$ is noncompact CMC hypersurface in manifold with bounded geometry.
		
	\end{proof}

	\begin{remark}
		In general, whether finite index of CMC hypersurface directly implies finite topology or that $H_2^1(M)$ is of finite dimension is not completely known. For minimal hypersurfaces in Euclidean space $\mathbb{R}^{n+1}$, this is true and proved by P. Li and J. Wang \cite{Li-Wang-finiteindex} and C. Li \cite{lichaoindexestimate} in $\mathbb{R}^4$. 
	\end{remark}
	
	\begin{remark}
		It follows from the proof and Lemma \ref{boundary} that if $M$ has finite first $L^2$-Betti number, we can obtain the same result.
	\end{remark}

	\section{A question in hyperbolic space}\label{last section}
	In this section we raise a question that is related to the topic of this paper. The question is to \textit{classify complete noncompact weakly stable constant mean curvature hypersurface in $\mathbb{H}^{4}$ with the mean curvature $H\geq 1$.} 
	
	In the following we provide some motivations.
	Although the weak stability (see the introduction) is usually defined for nonzero constant mean curvature hypersurface while stability is defined for minimal hypersurface, it still makes sense to define weakly stable minimal hypersurface. We turn our discussion back to Euclidean space first. The generalized stable Bernstein problem in $\mathbb{R}^{n+1}$  can be stated as: Classify complete noncompact weakly stable two-sided constant mean curvature hypersurface in $\mathbb{R}^{n+1}$ for $3\leq n+1\leq 7$. This can be divided into two parts. 
	\begin{itemize}
		\item[(1)] do Carmo's question: Complete weakly stable (or more generally finite index) constant mean curvature hypersurface with nonzero mean curvature must be compact ? As mentioned in the introduction, this was proved  for $n+1=3,4$.
		\item[(2)] Complete noncompact weakly stable two-sided minimal hypersurface in $\mathbb{R}^{n+1}$ for $3\leq n+1\leq 7$ must be hyperplane ? 
	\end{itemize}
	
	In fact, we have the following result that follows from  O. Chodosh and C. Li\cites{chodoshliR4,chodoshliR4anisotropic},  G. Catino, P. Mastrolia and A. Roncoroni \cite{catino}, O. Chodosh, C. Li, P. Minter and D. Stryker\cite{chodoshliR5} as well as L. Mazet\cite{mazet}.
	
	\begin{theorem}
		Let $M^n$ be a weakly stable complete noncompact CMC hypersurface (two-sided if it is minimal) in $\mathbb{R}^{n+1}$ for $n=3,4,5$. Then $M$ must be a hyperplane. In particular, complete weakly stable CMC hypersurface with nonzero mean curvature in $\mathbb{R}^6$ must be compact.
	\end{theorem}
	\begin{proof}
		We show that the curvature estimate 
		\[|A_M|(x)d_M(x,\partial M)\leq C\]
		holds for weakly stable compact CMC hypersurface $M$. Otherwise, we use the standard point-picking argument to get a limiting hypersurface $M_\infty$ which is a complete noncompact, two-sided, weakly stable minimal immersion into $\mathbb{R}^{n+1}$ with $|A_{M_\infty}|(0)=1$ and the bounded second fundamental form. It is known that the area of $M_\infty$ is infinity and by the Gauss equation, 
		\[\operatorname{Ric}_{M_\infty}(e_1,e_1)=-\sum_{i=1}^{n}h^2_{1i},\]
		so the Ricci curvature of $M_\infty$ has a lower bound. By the Bishop-Gromov volume comparison theorem, we have that $|B_{R+1}|\leq C|B_R|$ (see the remark in \cite[Page 440]{Barbosa-Berard-twisted-eigenvalue}). Thus according to L. Barbosa and P. Berard \cite[Theorem 3.10]{Barbosa-Berard-twisted-eigenvalue}, the weak index and strong index of $M_\infty$ coincide, so $M_\infty$ must be stable. Hence, by stable Bernstein theorem mentioned above, $M_\infty$ is a hyperplane, contradicting with the second fundamental form at origin.
		
		The curvature estimate by letting $\partial M$ diverge to infinity implies that $M$ must be hyperplane. 
	\end{proof}
	
	The question raised in the hyperbolic space at the beginning of this section can be analyzed in a similar manner. We partially answer this question in the current paper for $H>1$.

	\bibliography{refs}
	
\end{document}